\documentclass[11pt]{amsart}
\usepackage[T1]{fontenc}
\usepackage{amssymb, amsthm, amsmath}
\usepackage{dsfont}
\usepackage[dvipdf]{graphicx}%
\usepackage{hyperref}

 \DeclareMathOperator{\acl}{acl}
\DeclareMathOperator{\dcl}{dcl}

 \DeclareMathOperator{\tp}{tp}

\DeclareMathOperator{\stp}{stp}

\newtheorem*{nthm}{Theorem}

\newtheorem{theorem}{Theorem}[section]

\newtheorem{claim}[theorem]{Claim}

\newtheorem{conjecture}[theorem]{Conjecture}
\newtheorem{corollary}[theorem]{Corollary}

\newtheorem{fact}[theorem]{Fact}
\newtheorem{lemma}[theorem]{Lemma}

\newtheorem{proposition}[theorem]{Proposition}

\theoremstyle{definition}
\newtheorem{definition}[theorem]{Definition}

\newtheorem{remark}[theorem]{Remark}
\newtheorem{question}[theorem]{Question}

\newcommand{\Qq}{{\mathds{Q}}}

\newcommand{\CN}{{\mathcal N}}

\newcommand{\CM}{{\mathcal M}}

\newcommand{\CF}{{\mathcal F}}

\newcommand{\x}{\bar x}

\newcommand{\0}{\emptyset}
\newcommand{\vphi}{\varphi}

\renewcommand{\phi}{\varphi}
\renewcommand{\a}{\bar a}
\renewcommand{\b}{\bar b}
\renewcommand{\x}{\bar x}

\long\def\symbolfootnote[#1]#2{\begingroup%
\def\thefootnote{\fnsymbol{footnote}}\footnote[#1]{#2}\endgroup}

\makeatletter

\def\Ind#1#2{#1\setbox0=\hbox{$#1x$}\kern\wd0\hbox to 0pt{\hss$#1\mid$\hss}
\lower.9\ht0\hbox to 0pt{\hss$#1\smile$\hss}\kern\wd0}
\def\ind{\mathop{\mathpalette\Ind{}}}
\def\Notind#1#2{#1\setbox0=\hbox{$#1x$}\kern\wd0\hbox to 0pt{\mathchardef
\nn=12854\hss$#1\nn$\kern1.4\wd0\hss}\hbox to
0pt{\hss$#1\mid$\hss}\lower.9\ht0 \hbox to
0pt{\hss$#1\smile$\hss}\kern\wd0}
\def\nind{\mathop{\mathpalette\Notind{}}}

\def\thind{\mathop{\mathpalette\Ind{}}^{\text{\th}} }
\def\nthind{\mathop{\mathpalette\Notind{}}^{\text{\th}} }
\def\uth{{\rm{U}^{\text{\th}}}}
\def\ur{{\rm{U}}}
\def\tho{\text{\th}}

\DeclareMathOperator{\Tho}{\tho}

\makeatother

\title{Minimal types in super-dependent theories}
\author[A. HASSON]{Assaf Hasson$^*$}
\thanks{$^*$Supported by the EPSRC grant no. EP C52800X 1}
\address{$^*$Mathematical Institute\\
Oxford University\\
Oxford\\
UK} \email{hasson@maths.ox.ac.uk} \urladdr{http://www.maths.ox.ac.uk/\textasciitilde hasson/}
\date{\today}

\author[A. Onshuus]{Alf Onshuus}

\address{Universidad de los Andes,
Departemento de Matem\'aticas, Cra. 1 No 18A-10, Bogot\'{a},
Colombia} \email{onshuus@gmail.com}

\begin{document}

\begin{abstract}
We give necessary and sufficient geometric conditions for a theory
definable in an o-minimal structure to interpret a real closed
field. The proof goes through an analysis of \th-minimal types in
super-rosy dependent theories of finite rank. We prove that such
theories are coordinatised by \th-minimal types and that such a
type is unstable if an only if every non-algebraic extension
thereof is. We conclude that a type is stable if and only if it
admits a coordinatisation in \th-minimal stable types. We also
show that non-trivial \th-minimal stable types extend stable sets.
\end{abstract}

\maketitle

\section{introduction}
In this paper we complete the geometric classification of unstable
structures definable in o-minimal theories started in \cite{HaOn}, giving necessary and
sufficient geometric conditions for such a theory to interpret a
real closed field. We also analyse the analogous problem in the
stable case, reducing it to Zilber's Trichotomy
conjecture for strongly minimal sets interpretable in o-minimal
theories. Our main result can be summarised by:
\begin{nthm}Let $\CM$ be a structure definable in an o-minimal theory. Then:
\begin{enumerate}
 \item  $\CM$ interprets a real closed field if and only if there is
 a minimal non-locally modular (non trivial) global type.

\item Assuming Zilber's Trichotomy conjecture $\CM$ interprets a
pure algebraically closed field if and only if it has a minimal
non-locally modular stable type. \item If $p\in S(M)$ is minimal,
non-trivial and locally modular then for some $\vphi(x)\in p$ and
$\CM$-definable equivalence relation $E$ with finite classes the
induced structure on $\vphi(x)/E$ is interpretable in an (ordered,
if $p$ is unstable) vector space over an (ordered) division ring.
\end{enumerate}
\end{nthm}

It turns out that the right context for treating these questions is
that of super-rosy dependent (super-dependent, for short) theories
of finite rank. Motivated by the (by now classical) work on
superstable theories of finite rank our first goal is to develop a
theory for \th-minimal types.  This approach builds heavily on the
fact that super-dependent theories of finite rank admit a
coordinatisation in minimal types (Theorem \ref{coordinatization}).

The first task, taken care of in Section \ref{stable section}, is
to develop an appropriate notion of stability. As easy examples
demonstrate (see e.g. \cite{HaOn}) stability of definable sets is
too strong a notion to be of much use in the present setting,
and has to be identified on the level of types.
\emph{Hereditarily stable types}, introduced in \cite{HaOn}
(though in the context of dependent theories they coincide with
the stable types of Lascar and Poizat in \cite{LaPo}), seem to
give the right dividing line. The key to our analysis is the
observation that for hereditarily stable types forking coincides
with \th-forking. The technical key to the present paper is,
however:

\begin{nthm}
Let $p(x,a)\in S(Aa)$ be a stable type in a model of a dependent
rosy theory and assume that $p(x):=p(x,a)\upharpoonright_A$ is
unstable. Then $p(x,a)$ \th-forks over $A$.
\end{nthm}
This theorem assures (among others) that any non-algebraic
extension of a \th-minimal unstable type is unstable. The
importance of this last fact in the present context is easier to
understand when combined with the results from \cite{HaOn}. To
have a cleaner description of these results it will be convenient
to introduce:

\begin{definition}\label{o-minimal almost extension}
Let $p\in S(A)$ and $B\supseteq A$. A type $q\in S(B)$ is an
\emph{almost extension of $p$} if there is some $p\subseteq q'\in
S(B)$ and a $B$-definable equivalence relation $\sim$ with finite
classes such that $q=q'/\sim$.
\end{definition}

Recall that a type \emph{$p$ is o-minimal} if it extends a
definable set $S$ linearly ordered by a definable relation,
$<$, and o-minimal with respect to $<$ together with the induced
structure on $S$; it is \emph{almost o-minimal} if it is
\th-minimal and has an o-minimal almost extension. In those terms,
Theorem 2 of \cite{HaOn} asserts that in a structure definable in
an o-minimal theory an unstable type has either an almost o-minimal
extension or a hereditarily stable extension. Joining these two
results together we get that a \th-minimal unstable type $p$ is
almost o-minimal. Now in order to interpret a field it suffices to
show that this o-minimal structure is rich. As could be expected,
this will happen precisely when $p$ is a \emph{rich type} (namely,
there exists a large family of \th-curves in $p^2$).

\medskip

Somewhat more complicated than the analysis of the unstable case
is the treatment of the stable case. Equipped with the ideology -
and the analogy with stable theories - that in the stable case geometric complexity is
compensated by model theoretic simplicity, we prove:
\begin{nthm}
Let $p\in S(A)$ be a stable, non-trivial, minimal type. Then $p$
extends a definable stable stably embedded set.
\end{nthm}
When $p$ is rich the proof is a straightforward adaptation of
Buechler's Dichotomy theorem. In the locally modular case the
proof is slightly more complicated: we adapt Hrushovski's theorem
showing that for a locally modular minimal hereditarily stable
type $p$ there exists a minimal hereditarily stable $s\not \perp
p$ and a type-definable minimal group $G_0$ acting generically on
$s$. Pursuing our generalisation of Hrushovski's theorem, we show
that $G_0$ has a definable minimal supergroup and that $G$ has a
hereditarily generic stable type. From this we conclude that $G$
is a stable stably embedded group. This work takes the better part
of Section \ref{geometry-section}.

\subsection{\th-forking preliminaries.}

Before proceeding we remind the basic \th-forking
terminology needed to define the setting we will be working in; all
definitions and facts are taken from \cite{On} and \cite{onshuusthesis}.

\begin{definition}\label{thorking}
A formula $\delta (x,a)$ \emph{strongly divides over $A$} if
$\tp(a/A)$ is non-algebraic and $\{ \delta (x,a') \}_{a'\models
\tp(a/A)}$ is $k$-inconsistent for some $k\in \mathbb{N}$. A
formula $\delta (x,a)$ \emph{\th-divides over $A$} if there is a
tuple $c$ such that $\delta (x,a)$ strongly divides over $Ac$.
\end{definition}

Standard forking terminology generalises naturally to \th-forking.
In particular, a formula \th-forks over a set $A$ if it implies a
finite disjunction of formulas \th-dividing over $A$.

What makes this notions useful
is the existence of a large class of theories, rosy theories,
where non-\th-forking
(\th-independence) gives rise to a geometric notion of
independence so, in particular, symmetric and transitive. The class of
rosy theories include, among others, all simple theories, and
more importantly for our purposes, o-minimal theories and any
theory they interpret.

In stable theories \th-independence coincides with
forking-independence. In fact, the following stronger result
(Theorem 5.1.1 of \cite{On}) that we will need is true in any rosy
theory:

\begin{fact}\label{5.1.1}
If $\phi(x,y)$ satisfies NOP and there is a $\phi$-formula
witnessing that $\tp(a/bc)$ forks over $c$, then there is a
$\phi$-formula witnessing that $\tp(a/bc)$ \th-forks over $c$.
\end{fact}

We will make ample use of the \th-forking analogues of global
ranks defined in super simple theories: we define the \th-rank of
a formula to be the analogue of the global rank in simple
theories, so that \th$(\varphi(x,b))\geq \alpha+1$ if there is
$\psi(x, c) \vdash \varphi(x, b)$ \th-dividing over $b$ with
\th$(\psi(x,c))\geq \alpha$.  We also define the $\uth$-rank as
the foundation rank of the partial order (on complete types) with
$p<_\tho q$ defined as ``$p$ is a \th-forking extension of $q$''.
In analogy with the case of simple theories, a theory where all
types have ordinal valued $\uth$-rank is called super-rosy. Recall
(see e.g. \cite{HaOn}) that in an o-minimal theory
$\tho(\phi(x))=\dim \phi(x)$ for any formula $\phi(x)$; in
particular, $\uth(p)=\dim p$ for all $p$. It follows that if $\CN$
is interpretable in an o-minimal theory, $\CN$ is super-rosy of
finite $\uth$-rank and thus \th-independence is a geometric
independence relation.

\medskip

\noindent We can now introduce the main object of investigation in
this paper:

\begin{definition}
A type $p$ is \emph{minimal} if $\uth(p)=1$. A formula $\theta$ is
\emph{minimal} if $\Tho(\theta)=1$. A type $p$ is \emph{seriously
minimal} if it extends a minimal set.
\end{definition}

Our definition of a minimal formula is the analogue of
\emph{weakly minimal} formulas in the stable context. Since
minimal (in stability theoretic terminology) formulas are
(\th-)minimal, we have no remorse overriding this terminology.
As for minimal types, we will see that stable minimal
types (over algebraically closed sets) are stationary of
$\ur$-rank 1, and therefore minimal in the stability theoretic
sense as well.

\subsection{Coordinatisation}

The aim of this subsection is to show that in super-rosy theories
of finite \th-rank (and so, in particular those theories
interpretable in an o-minimal structure) minimal types control the
structure of the theory, justifying our focusing of the analysis
on them. We will prove that such theories are coordinatised (in
the sense of \S 4 of \cite{HaKiPi}) by \th-minimal types. Our
theorem slightly strengthens a similar unpublished result in
\cite{OnUs2}.

\medskip

\noindent For the proof we need a few easy observations, the first
of which can be found in \cite{HaOn}:

\begin{fact}\label{frank}
Let $\CN$ be definable in an o-minimal structure $\CM$, let
$\phi(x,b)$ be $\CN$-definable and let $p(x)\in S_n^{\CN}(N)$.
Then both $\tho (\phi(x,b))$ and $\uth(p(x))$ are finite.
\end{fact}

From the proof of Lemma 4.7 in \cite{HaOn} we get the following:

\begin{fact}\label{0}
Let $a,b$ be elements and $A$ be a set such that $\tp(a/Ab)$
\th-forks over $A$. Then there is some $\bar b$ such that
$a\thind_{Ab} \bar b$ and such that $\tp(a/Aa\bar b)$ \th-divides
over $A\bar b$.
\end{fact}

\noindent and

\begin{fact}\label{1}
Let $a,b$ be elements and $A$ be a set such that $\tp(a/Ab)$
\th-divides over $A$. Then there is some $e$ such that
$a\thind_{Ab} e$ and such that $\tp(a/Abe)$ strongly divides over
$Ae$. In particular, $\uth(a/Abe)<\uth(a/Ae)$.
\end{fact}

We will also need the following result which, although surprising,
follows immediately from the definition:

\begin{fact}\label{strong dividing}
Let $a,b,A$ be such that $a\not\in \acl(A)$. Then $\tp(b/Aa)$
strongly divides over $A$ if and only if  $a\in \acl(Ab')$ for all $b'\models
\tp(b/Aa)$.
\end{fact}

%
%
%

The following is the technical key to coordinatisation:

\begin{lemma}\label{technical coordinatisation}
Let $p(x)=\tp(a/A)$ be any type such that $\uth(p)=n$, let $b$ be
such that $\tp(a/Ab)$ \th-divides over $A$ and let $\uth(a/Ab)=m$.
Then there is some $e$ such that $a\thind_A e$, $a\thind_{Ab} e$
and $\tp(a/Abe)$ strongly divides over $Ae$.
\end{lemma}

\begin{proof}
We proceed by induction on $n-m$. For the case $n-m=1$, Fact \ref{1}
supplies us with some $e$ such that $a\thind_{Ab} e$ and $\tp(a/Abe)$
strongly divides over $Ae$, since also $\uth(a/Abe)<\uth(a/Ae)$ we get:
\[
n-1=\uth(a/Ab)=\uth(a/Abe)<\uth(a/Ae)\le \uth(a/A)=n
\]
implying that $\uth(a/Ae)=\uth(a/A)$, and the conclusion follows.

Let $a,b,A$ be as in the statement of the lemma. Let $e$ be as in
Fact \ref{1}. If $a\thind_A e$ we have nothing more
to do, so assume $a\nthind_A e$.

\begin{claim}\label{3}
We may assume that $\tp(a/Ae)$ \th-divides over $A$.
\end{claim}

\begin{proof}
By definition of \th-forking  $\tp(a/Ae)\vdash \bigvee_{i=1}^k
\phi_i(x,e_i)$ with each $\phi_i(x,e_i)$ \th-dividing over $A$.
Let $e':=\{e_1,\dots e_k\}$; note that if $e'\equiv_{Ae} f$ then
$\tp(a/Ae)\vdash \bigvee_{i=1}^k \phi_i(x,f_i)$ and each
$\phi_i(x,f_i)$ \th-divides over $A$. Hence we may assume that
$ab\thind_{Ae} e'$ so in particular, by transitivity,
$a\thind_{Ab} ee'$.

By definition $p(x,b,e):=\tp(a/Abe)$ strongly divides over $Ae$ if
and only if
\[
\{p(x,b',e)\}_{b'\models \tp(b/Ae)}
\]
is $k$-inconsistent for some $k$ and $b\not\in \acl(Ae)$. Thus,
\[\{p(x,b',e,e')\}_{b'\models \tp(b/Aee')}\] is
$k$-inconsistent. By monotonicity $b\thind_{Ae} e'$  so $b\notin
\acl(Aee')$ implying that $\tp(a/Abee')$ strongly divides over
$Aee'$.

Setting $\bar e:= ee'$ we get that $a\thind_{Ab} \bar e$,
$\tp(a/Ab\bar e)$ strongly divides over $A\bar e$ and $a$
\th-divides over $A$, as required.
\end{proof}

From now on we assume the conclusion of the claim to hold
and that $\tp(a/Ae)$ \th-divides over $A$. By assumption $\tp(a/Abe)$
strongly divides over $Ae$ and $a\thind_{Ab} e$, so
\[
 \uth(a/A)>\uth(a/Ae)>\uth(a/Abe)=m
\]
and $\uth(a/A)-\uth(a/Ae)<n-m$; by the inductive hypothesis there is
some $f$ such that $a\thind_{Ae} f$, $a\thind_A f$ and
$\tp(a/Aef)$ strongly divides over $Af$. Notice that the same is
true  if we replace $f$ with $f'\equiv_{Abe} f$, whence we may
assume that $ab\thind_{Ae} f$.

In particular $b\not\in \acl(Aef)$, so it follows that $\tp(a/Abe)$
strong divides over $Aef$. We will prove that $\tp(a/Abf)$ strong
divides over $Af$, thus completing the proof. By Fact \ref{strong
dividing} it is enough to show that $b\in \acl(Aa'f)$ whenever $a'\models
\tp(a/Abf)$.

Let $a'\models \tp(a/Abf)$ and let $e'$ be such that $e'a'\models
\tp(ea/Abf)$; by hypothesis $\tp(a'/Ae'b)$ strongly divides over
$Ae'$ and $\tp(a'/Ae'f)$ strongly divides over $Af$. By Fact
\ref{strong dividing} $e'\in \acl(Aa'f)$ and $b\in \acl(Aa'e')$; by
transitivity $b\in \acl(Afa')$ which completes the proof.
\end{proof}

\begin{definition} Given a set of types $\mathcal{P}$ closed under automorphisms,
$\tp(a/A)$ is \emph{coordinatised by $\mathcal{P}$} if there is a
sequence of elements $\langle a_0, \dots, a_n\rangle$ with $a_n=a$
and $\tp(a_{i+1}/Aa_1\dots a_i)\in \mathcal P$.
\end{definition}

\begin{theorem}\label{coordinatization}
Let $T$ be a rosy theory,  $p(x):=\tp(a/A)$ a type of $\uth$-rank
$n$. Then there is a non \th-forking extension $q:=\tp(a/E)$ such
that $q$ is coordinatised by minimal types. Moreover, we can
choose the coordinatising sequence $\langle a_0, \dots,
a_n\rangle$ such that $a_i\in \acl(aE)$.
\end{theorem}

\begin{proof}
Let $\tp(a/A)$ be as in the statement and let $b$ be any element such that
$p(x,b)=\tp(a/Ab)$ has $\uth$-rank 1. Now either $\uth(p)=1$, in
which case we have nothing to do or, by Lemma \ref{technical
coordinatisation}, we can find $e$ such that $\tp(a/Abe)$ is a
non \th-forking extension of $\tp(a/Ab)$ which strongly divides
over $Ae$ and such that $a\thind_A e$.

Since $b\in \acl(Aea)$  we know by Lascar's inequalities that
$\uth(b/Ae)=\uth(a/A)-1$ so, inductively,  $\tp(b/Ae)$ has a non
\th-forking extension $\tp(b/Aef)$ which can be coordinatised by
\th-minimal types with a coordinatising sequence in $\acl(bef)$.
By hypothesis $b\thind_{Ae} f$ so we can choose
$f$ so that $ab\thind_{Ae} f$.

Since $a\thind_A e$ by transitivity $a\thind_{A} ef$ and any
coordinatisation of $\tp(b/Aef)$ extends through $\tp(a/Abef)$ to
a coordinatisation of $\tp(a/Aef)$. Since we have $b\in \acl(ae)$
and by assumption we can coordinatise $\tp(b/Aef)$ with elements
from $\acl(Aefb)$, the theorem follows.
\end{proof}


\section{Hereditarily stable types and \th-forking}\label{stable section}

This section is dedicated to understanding stable types in
dependent theories. Though several of the results seem to hold
without any assumptions on the theory $T$, we will assume
throughout that $T$ is dependent (unless specifically stated
otherwise).  We will use standard stability theoretic definitions
and notation from \cite{Sh} (mainly Section II) and \cite{Pi}.
The following slightly strengthens Theorem II-2.2 in \cite{Sh}:

\begin{fact}\label{stable types shelah}
Let $T$ be any theory (not necessarily dependent). For a type
$p\in S(A)$ and a formula $\phi(x,y)$ over $C\supset A$ the
following are equivalent.

\begin{enumerate}

\item There exists $B\supseteq C$ such that
\[
|\{q\in S_{\phi}(B): p\cup q \text{ is consistent}\}|>|B|
\]

\item There are infinite indiscernible sequences $\langle
a_i\rangle_{i\in I}$ in $p$ and $\langle b_i\rangle_{i\in I}$ such
that $\models \phi(a_i,b_j)$ if and only if $i<j$.

\item $\Gamma_p(\phi,n)$ is consistent for any $n\in \mathbb{N}$, where
\[
 \Gamma_p(\phi,n):=\{p(x_{\eta})\mid \eta\in 2^n\}\cup \{\phi(x, y_{\eta[k]})^{\eta[k]} \mid \eta\in 2^n, k<n\}
\]
\item $R(p,\phi,\infty)=\infty$.
\item $R(p,\phi,2)=\omega$.
\end{enumerate}
Where $R(p,\phi,-)$ are Shelah's local ranks of \cite{Sh}, II.1.
\end{fact}
\begin{proof}
\noindent (2)$\Rightarrow$(1) is standard. \\
\noindent (1)$\Rightarrow$(2) follows from the aforementioned
Theorem II-2.2 in \cite{Sh} as follows: enrich the language by
constants for $A$. Fix an arbitrary $\theta(x)\in p$ and replace
$\phi(x,y)$ by $\phi':=\phi(x,y)\land \theta(x)$. By assumption
$S_{\phi'}(B)$ is large for some $B$. So by Shelah's theorem there
is an infinite sequence witnessing that $\phi'$ has OP. But
$\models \phi'(a_n,b_m)$ for $m>n$ implies that $\models
\theta(a_n)$ for all $n$. By compactness the claim follows.

\noindent The rest are easy adaptations of Lemmas 2.3-2.9(1) in
Chapter II of \cite{Sh}
\end{proof}

Following \cite{LaPo} we define a type $p$ to be \emph{stable} if
$p$ does not satisfy any of the above conditions for any
$\phi(x,y)$. Also following \cite{LaPo}, we will say that a definable (or $\infty$-definable) set
$X$ is \emph{stable} if there are no infinite
indiscernible sequences $\langle a_i\rangle_{i\in I}$ in $X$ and
$\langle b_i\rangle_{i\in I}$ such that $\models \phi(a_i,b_j)$ if
and only if $i<j$. An easy application of compactness shows that
a set is stable if and only if every type extending it is stable.

Recall that in \cite{Ke} Keisler defined \emph{the stable part of
$T$} to be the union of all those types whose
$\Delta$-Cantor-Bendixon rank is finite for all finite $\Delta$.
It may be worth noting that a stable type need not belong to the
stable part of the theory. Consider the structure $(\Qq,<^*)$
given by
\begin{equation*}
x<^*y\iff
\begin{cases} \text{$x<y<\pi$ or}
\\
\text{$\pi \le x,y\le 2\pi$ or }
\\
\text{$2\pi > x > y$}
\end{cases}
\end{equation*}
Take a saturated model, and consider $\tp(a/\Qq)$ for some
$a\models 1/2<x<1/2$. Clearly this last formula is stable and
stably embedded so $p$ is a stable type. But setting
$\Delta:=\{x<^*y\}$ any finite boolean combination of formulas of
the form $x<^* b$ has infinite $<^*$-chains so $p$ has infinite
$\Delta$-rank.

We remind that for a formula $\phi(x,y)$ (over $\0$) a type $p\in
S(A)$ is \emph{$\phi$-definable} if there is a formula, denoted
${\rm d_p}x\phi(x,y)$, such that $\phi(x,a)\in p$ if and only if
$\models {\rm d_p}x\phi(x,a)$ for all $a\in A$. Naturally, $p$ is
\emph{definable} if it is $\phi$-definable for every $\phi$. With
this terminology the following is Theorem 4.4 of \cite{LaPo}.

\begin{fact}
Let $p\in S(M)$ for some $M\models T$. The following are
equivalent.

\begin{enumerate}
\item $p$ is stable.
\item $p$ has at most $|N|^{\aleph_0}$ extensions to every model $N\succ M$ of $T$.
\item Every extension of $p$ to a model $N\succ M$ of $T$ is
definable.
\end{enumerate}
\end{fact}

\medskip

In \cite{HaOn} we defined a type $p\in S(A)$ to be
\emph{hereditarily stable} if there is no $\phi(x,y)$ (not
necessarily over $A$) defining a partial (quasi) order with
infinite chains in $p$. The next fact shows that that in
dependent theories being hereditarily stable is equivalent to being stable.

\begin{fact}\label{stable types}
Let $T$ be a dependent theory, $A\subseteq M\models T$.  For
$p\in S(A)$ the following are equivalent:
\begin{enumerate}
\item $p$ is not hereditarily stable

\item There is an indiscernible sequence in $p$ which is not an
indiscernible set.

\item There is $\phi(x,y)$ and infinite indiscernible sequences
$\langle a_i\rangle_{i\in I}$ in $p$ and $\langle b_i\rangle_{i\in
I}$ such that $\models \phi(a_i,b_j)$ if and only if $i<j$.
\end{enumerate}
\end{fact}

\begin{proof}
\noindent (1)$\Rightarrow$(2) is obvious from the definition. \\
\noindent (2)$\Rightarrow$(3) is standard
(and easily adapts from, e.g. the proof of Theorem II-2.13 of \cite{Sh}). \\
\noindent (3)$\Rightarrow$(1) is due to Shelah, see e.g. Lemma 4.1 of \cite{OnPe}.
\end{proof}

\begin{corollary}
In a dependent theory, a type is hereditarily stable if and only
if it is stable.
\end{corollary}

It should be clear (from our terminology - but also from the
definition) that any extension of a hereditarily stable type is
hereditarily stable. It follows from (1) of Fact \ref{stable
types shelah} that an almost extension of a hereditarily stable
type is hereditarily stable.

\medskip

Since we work in a dependent theory, we will not make any
distinction between ``hereditarily stable'' and ``stable'' types,
referring to both as ``stable types'' and freely using any of the
equivalent properties appearing in the previous facts. We point
out that this terminology is a strict strengthening of Shelah's
definition of stable types in \cite{Sh1}, but since Shelah-stable
types will not be used herein, no confusion can arise. In this
spirit, we will define a type to be \emph{unstable} if it is not
(hereditarily) stable.

The following is useful and we believe it should be known, but we
could not find a reference, so we include the proof.

\begin{lemma}\label{unstable to stable forks} Let $p_0(x)\in S(A_0)$ be a
stable extension of a type $p(x)\in S(A)$ (for some $A\subseteq
A_0$) which is unstable. Then $p_0(x)$ forks over $A$.
\end{lemma}

\begin{proof}
By Fact \ref{stable types} we know there is a set $B\supseteq
A$, a $B$-indiscernible sequence $\langle b_i \rangle_{i\in
\mathbb{Q}}$ of realizations of $p$ and a $B$-definable transitive
relation $\leq$ such that $b_i\leq b_j$ if and only if $i\leq j$.
Let $b\models p_0$. Since $\tp(b_i/A)=\tp(b/A)=p$, we may assume -
possibly changing $B$ - that $b_{1/2}=b$ so that $b_0\leq x\leq
b_1$ is consistent with $p(x,a)$.

Because $p_0$ is stable there can only be finitely many $n\in
\mathbb{Z}$ for which $\{b_n\leq x\leq b_{n+1}\}\cup p_0(x)$ is
consistent (otherwise, for each $n$ such that it is consistent we
can find $d_n$ witnessing it, showing that $\le$ has infinite
chains in $p_0$).  It follows that $p_0(x)\vdash b_n\leq x\leq
b_{n+r}$ for some $n\in \mathbb{Z}, r\in \mathbb{N}$. Since $\le$
is a quasi order and the $b_i$ form a $\le$-chain it follows that
\[
\{ b_{n+kr}\leq x\leq b_{n+(k+1)r}\}_{r\in \mathbb{Z}}
\]
is 2-inconsistent showing that
$p_0(x)$ divides (and thus forks) over $A$.
\end{proof}

The following claim is an easy application of compactness and Fact
\ref{stable types}, left as an exercise:

\begin{claim}\label{NOP}
Let $p\in S(A)$ be a stable type and $\phi(x,y)$ any formula. Then
there exists $\theta(x)\in p$ such that $\phi(x,y)\land \theta(x)$
has NOP.
\end{claim}

\begin{lemma}\label{th-forking extensions}
Let $q(x)\in S(B)$ be an extension of a stable type $p(x)\in S(A)$
with $A\subset B$. Then, $q(x)$ forks over $A$ if and only if
$q(x)$ \th-forks over $A$.
\end{lemma}

\begin{proof}
Any instance of \th-forking is, by definition, an instance of
forking. So we only need to prove the other direction.

Assume $q\supseteq p$ is a forking extension. By definition there
is some $\phi(x,b)\in q$ forking over $A$. Let $\theta(x)\in p(x)$
be as provided by the previous claim. Then $\theta(x)\land
\phi(x,y)$ forks over $A$ and, since it also has NOP, it follows
by Fact \ref{5.1.1} that $q$ is a \th-forking extension of $p$, as
required.
\end{proof}

\begin{corollary}\label{preserving th-rank}
Let $p(x)\in S(A)$ be a stable type. If $p\in S(A)$ has ordinal
valued $\uth$-rank then $\uth(p)=\ur(p)$. In particular, if $p$ is
stable and \th-minimal then $\ur(p)=1$.
\end{corollary}

\begin{proof}
This follows immediately from Lemma \ref{th-forking extensions}
and the definitions of $\ur$-rank and $\uth$-rank.
\end{proof}

The upcoming proofs involve some local results
from stability theory. The following are, respectively, Lemma 2.10
and Lemma 2.11 in \cite{Pi-fo}.

\begin{fact}\label{1.2.17'} \label{1.2.16}
Let $\delta(x,y)$ be a formula satisfying NOP, let $q(x)\in
S_{\delta}(B)$ and let $p(x)=q(x)\upharpoonright A$ for some
$A\subseteq B$. Then the following hold:

\begin{description}
\item[Definability of non forking extensions] $q(x)$ does not
fork over $A$ if and only if the $\delta$-definition of $q(x)$
is almost  over $A$.

\item[Open mapping theorem] $q(x)$ does not fork over $A$ if and
only if for every $\theta(x,b)\in q(x)$ there is some
$\sigma(x)\in p(x)$ which is a positive boolean combination of
$A$-conjugates of $\theta(x,a)$.

\end{description}
\end{fact}

In particular, we can strengthen Corollary \ref{preserving
th-rank}:

\begin{lemma}\label{preserving th-rank 2}
Let $p\in S(A)$ be a stable type in a dependent rosy theory, $q\supseteq p$ a non-forking extension. Then $\tho(p)=\tho(q)$. If in addition $\Tho(p)=\alpha<\infty$ then converse is also true.
\end{lemma}

\begin{proof}
Let $\Tho(q)=\alpha$ (possibly $\alpha=\infty$). By
definition there is $\phi(x,b)\in q$ with
$\Tho(\phi(x,b))=\alpha$; by Claim \ref{NOP} we may assume
$\phi(x,y)$ has NOP. Thus, by Fact \ref{1.2.17'} there is a
formula $\sigma(x)\in p$, which is a positive Boolean combination
of $A$-translates of $\phi(x,b)$. Since it is always true that
$\tho(\theta_1\lor \theta_2) = \max
\{\Tho(\theta_1),\Tho(\theta_2))\}$ and $\tho(\theta_1\land
\theta_2) \leq \min \{\Tho(\theta_1),\Tho(\theta_2))\}$ we must
have that $\Tho(\sigma) \le \Tho(\psi)$, and since $\Tho(p)\ge
\Tho(q)$ is always true equality follows.

For the other direction assume that $q$ was a forking extension of $p$. In that case we could find $\psi\in q$ witnessing it. Moreorver, we could choose $\psi$ so that $\Tho(\psi)=\alpha$ and, for any $\vphi\in p$ we could also require that $\psi\to \vphi$. In particular, we can choose $\vphi$ such that $\Tho(\vphi)=\alpha$. But on the other hand, by definition of $\Tho$-rank, all this would imply $\Tho(\vphi)\ge \alpha+1$, a contradiction.
\end{proof}

We can also show that stable types over algebraically closed sets
are stationary (Lemma \ref{definability} below). This was
generalised by Usvyatsov (\cite{Us}) and independently in
\cite{HrPi2} showing that Shelah-stable types over algebraically
closed sets are stationary. We include the proof for completeness:

\begin{lemma}\label{definability}
Let $A=\acl(A)$, $p\in S(A)$ stable, and $p\subseteq q\in S(B)$ a
non forking extension. Then $q$ is definable over $A$; in
particular $q$ is the unique non forking extension of $p$ to $B$.
\end{lemma}

\begin{proof}
Let $\phi(x,b)\in q$. By Claim \ref{NOP} there exists
$\theta(x)\in p$ such that $\theta(x)\land \phi(x,y)$ has NOP.
Note that $c \models {\rm d_p}x(\phi(x,y))$ if and only if $c\models
({\rm d_p}x) (\phi(x,y)\land \theta)$, so we may assume that $\phi(x,y)$
has NOP. The result now follows from Fact \ref{1.2.16}.
\end{proof}

In \cite{HaOn} we proved that in any theory interpretable in an
o-minimal structure, any unstable type $p$ either has an o-minimal
almost extension (see Definition \ref{o-minimal almost extension})
or has a non-algebraic stable extension; in particular, if every
stable almost extension of $p$ is algebraic, then the former
option must hold. In Theorem \ref{unstable to stable} we prove
that this is, in fact, the case with minimal unstable types in a
rosy theory (this will be discussed in more detail in Section
\ref{in o-minimal}).

\medskip

\noindent The key to the proof of Theorem \ref{unstable to stable} is: .

\begin{lemma}\label{non thorking definable}
Let $M$ be the model of a rosy theory, let $A=\acl(A)$ and
$p(x,a)\in S(Aa)$ be stable ($a$ possibly an infinite tuple). If
$p(x,a)$ does not \th-fork over $A$ then it is definable over
$\acl(A)$.
\end{lemma}

For the proof we need the following, which is Lemma 4.1.11 of
\cite{On}:

\begin{fact}\label{4.1.11}
Let $p(x,a)$ and $p(x,b)$ be non \th-forking extensions of
$p(x)\in S(A)$ with $a\thind_A b$. Then there is some $b'\models
\tp(b/A)$ such that $p(x,a)\cup p(x,b')$ does not \th-fork over
$A$ and $a\thind_A b'$.
\end{fact}

\begin{proof}[Proof of Lemma \ref{non thorking definable}]
Let $p(x,a)$ be as in the statement, $\phi(x,y)$ any
formula, and
$\psi(y,a):=(\rm{d_{p(x,a)}}x)\phi(x,y)$, the $\phi$-definition
of $p(x,a)$.

By Fact \ref{4.1.11} there is some $b\models \tp(a/A)$ such that
$a\thind_A b$ and $p(x,a)\cup p(x,b)$ does not \th-fork over $A$;
in particular it does not \th-fork over $Aa$ and therefore has a
completion $q\in S(Aba)$ which does not \th-fork over $Aa$. By
Lemma \ref{th-forking extensions} $q$ does not fork over $Aa$.

It follows from Lemma \ref{definability} that $q$
is definable over $\acl(Aa)$. By symmetry $q$ is also definable
over $\acl(Ab)$. Therefore $\models (\forall
y)(\psi(y,a)\iff \psi(y,b))$. So if we define
\[
z\simeq w  \iff   \forall y (\psi(y,z) \leftrightarrow \psi(y,w))
\]
we get that $[b]_\simeq\in \acl(Aa)$ and $[a]_\simeq\in \acl(Ab)$.
Since $\tp(a/Ab)$ is non algebraic (otherwise $a\in \acl(A)$ and
we have nothing to do) it has infinitely many realizations so $\{w
\mid a \simeq w\}$ is infinite.

Let $e:=[a]_\simeq$; if $e\in \acl(A)$ the lemma follows, so
assume towards a contradiction that $e\not\in \acl(A)$. By
definition the formula ``$[x]_{\simeq}=e$'' 2-\th-divides over $A$
so in particular it witnesses that $a\nthind_A e$. But $b\models
p$ with $a\thind_A b$ and $e\in \acl(Ab)$ so $a\thind_A
e$, a contradiction.  
\end{proof}

\begin{remark}
It may be tempting to try and prove the last lemma using Claim
\ref{NOP} in a stronger way than we have ($p(x,a)$ is stable and
non \th-forking over $A$ and therefore non-forking over $A$), but
we have to be careful: if $\phi(x,a)\in p(x,a)$  forks over $A$
(for some $\phi(x,y)$ over $A$) Claim \ref{NOP} provides us with a
formula $\theta(x)\in p(x,a)$ (that is over $Aa$) such that
$\theta(x)\land \phi(x,y)$ has NOP. This will not tell us anything
about the behaviour of $\phi(x,y)$ over $A$.

This shows that while Claim \ref{NOP} is useful when trying
to understand forking relations between a stable type and its
extensions, understanding definability and forking relations
between a stable type and its restrictions requires more subtlety.
\end{remark}

The main result of this section is now easy to prove, and will
be useful later on:

\begin{theorem}\label{unstable to stable}
Let $p(x,a)\in S(Aa)$ be a stable type in a model of a dependent
rosy theory and assume that $p(x):=p(x,a)\upharpoonright_A$ is
unstable. Then $p(x,a)$ \th-forks over $A$.
\end{theorem}

\begin{proof}
Let $p(x,a)$ be a stable type, $\phi(x,y)$ any formula over $A$
and assume towards a contradiction that $p(x,a)$ does not \th-fork
over $A$. By Lemma \ref{definability} $p(x,a)$ has a
$\phi$-definition over $\acl(Aa)$ and by Lemma \ref{non thorking
definable} this definition is over $\acl(A)$. So $\phi(x,b)\in p$
if and only if $b\models (\rm{d_p}x)\phi(x,y)$; since
$(\rm{d_p}x)\phi(x,y)$ is over $\acl(A)$ and $\phi(x,y)$ was
arbitrary it follows that $p(x,a)$ does not fork over $\acl(A)$.
But by assumption $p(x)$ is unstable. This contradicts Fact
\ref{unstable to stable forks}.
\end{proof}

As a first application of the above theorem we get:

\begin{theorem}
Let $p\in S(A)$ be any type of finite $\uth$-rank in a dependent
rosy theory. Then $p$ is stable if and only if some (equivalently,
any) non \th-forking $q\supseteq p$ is coordinatised by stable
types.
\end{theorem}

\begin{proof}
Let $a\models p$. By Theorem \ref{coordinatization} there is a non
\th-forking extension $\tp(a/E)$ of $p$ which is coordinatised by
a sequence $\langle a_0, \dots, a_n\rangle$ such that
$\tp(a_{i+1}/Aa_1\dots a_i)$ is minimal and $a_i\in \acl(Ea)$. It
is now easy to check (using Fact \ref{stable types shelah}) that
$\tp(a/E)$ is stable if and only if $\tp(a_{i+1}/Aa_1\dots a_i)$
is for all $i$. But Theorem \ref{unstable to stable} implies that
$\tp(a/E)$ is stable if and only if $\tp(a/A)$ is, so the theorem
follows.
\end{proof}

In particular we get:
\begin{corollary}
A super-dependent theory $T$ of finite rank is stable if and only
if every \th-minimal type in $T$ is stable (if and only if every
\th-minimal type in $T$ has $\ur$-rank one).
\end{corollary}

In a rosy theory the notion of orthogonality and hence also that
of \th-regular types are naturally defined. In super-rosy theories
every type is domination equivalent to the sum (in the appropriate
sense) of \th-regular types. It is therefore natural to ask:

\begin{question}
Let $T$ be a super-dependent theory. Is $T$ stable if and only if
every \th-regular type in $T$ is stable ?
\end{question}

\section{The geometry of minimal types.}\label{geometry-section}

In this section we develop the basic geometric theory of minimal
types. We show that minimal stable types behave in many ways like
minimal types in stable theories. This is explained by the main
result of this section, proving that a non-trivial minimal stable
type (as always, in a super-dependent theory) is seriously stable.

As usual, definable families of plane curves play a crucial
role in the analysis:

\begin{definition}
 Let $\theta$ be a minimal formula.
\begin{enumerate}
\item A plane curve (with respect to $\theta$) is a minimal subset
of $\theta^2$. If $\Phi\subseteq \theta^2$ is type-definable, a
plane curve $C(x,y)$ (with respect to $\theta$) is \emph{through}
$\Phi$ if there is some $\bar b\models \Phi\cup \{C(x,y)\}$ such
that $\bar b \notin \acl(A)$ for some $A$ over which $C(x,y)$ is
defined.

\item A definable family of plane curves $\{\theta_c:c\models
\psi\}$ is \emph{almost normal} at the ($\infty$-)definable
$\Phi\subseteq \theta^2$ if for all $c\models \psi$ there are
finitely many $c'\models \psi$ such that
$\{\theta_c\land\theta_{c'}\}\cup \Phi$ is non-algebraic.

\item If $\theta \in p$ then $\{\theta_c:c\models \psi\}$ is
\emph{almost normal at $p$} if it is almost normal at $p(x)\times
p(y)$.
\end{enumerate}
\end{definition}

\begin{definition}
 A minimal type $p\in S(A_0)$ is:
\begin{description}
\item[Trivial] if $b\in \acl(Aa_1,\dots,a_n)\iff b\in \bigcup
\acl(Aa_i)$ for every set of parameters $A_0\subseteq A$ and every
$a_1\dots,a_n,b\models p$.
\item[Locally modular] if there is a
set $C\supseteq A_0$ such that
\[
a_1\dots,a_n\ind_{A\cap B}
b_1\dots b_m
\]
for all $a_1\dots a_n$,$b_1\dots,b_m\models p$ with
$A=\acl(Ca_1,\dots,a_n)$ and $B=\acl(Cb_1,\dots,b_m)$. \item[Rich]
if it is seriously minimal and there exist a finite set
$A_0\subseteq A$, an element $b \models p$ with $b\thind_{A_0} A$
and an $A$-definable family $\CF$ of plane curves (almost) normal
at $p$,
 such that $\{f\in \CF : (b,b)\in f\}$ is infinite.
\end{description}

To make the statements in this paper cleaner ``locally modular''
will always mean ``locally modular and non-trivial''.
\end{definition}

We will use the following easy observations which we leave for the
reader to verify:

\begin{remark}\label{easy facts} Let $p$ be a seriously minimal type, then:
\begin{enumerate}
\item If $p$ is rich and $q\supseteq p$ is a non-algebraic
extension then $q$ is rich.

\item If $p\in S(A)$ is rich (resp. non-trivial) and $A_0\subseteq
A$ then $p\upharpoonright A_0$ is rich (non-trivial).

\item If $p$ is non-trivial it has a global non-trivial extension
(see also the proof of Theorem 3 of \cite{HaOnPe}).

\item Every non-algebraic extension of a non-trivial type is
non-trivial.
\end{enumerate}
\end{remark}

A locally modular type cannot be rich, but it is not clear to us
whether, in general, any minimal non-locally modular type is rich.
The problem lies in the existence of normal families in contexts
where there is no obvious candidate for the notion of \emph{germ}
(this problem does not exist for stable, or even Shelah-stable
types).

\begin{remark}\label{rich}
If $T$ is o-minimal then a global non trivial minimal type is rich
if and only if it is not locally modular.
\end{remark}
\begin{proof}
Note that for any $M\models T$ a type $p\in S_n(M)$ is minimal
if and only if $\dim p =1$ if and only if $p$ is seriously
minimal. Moreover, if $\theta(x)\in p$ is a minimal formula, we
may assume without loss that the set it defines is definably
isomorphic to an open interval in $M$. Translating everything
through this definable isomorphism, the analysis of minimal types
in $\CM$ reduces to the analysis of non-algebraic 1-types. Naming
a small model of $T$ we may assume that $\acl(\0)\models T$.

So let $p$ be a non trivial global 1-type, $p_0:=p\upharpoonright
\acl(\0)$ and $a\models p_0$. If $p$ is rich then by the previous
remark so is $p_0$ and therefore $a$ is of type (Z3) in the sense
of \cite{PeSt}, namely on every definable interval $I\ni a$ there
is a definable normal family of plane curves of dimension at least
2. In that case there is some interval $I\ni a$ and a
$T$-definable field on $I$, and $p$ cannot be locally modular. On
the other hand, if $p$ is not rich neither is $p_0$ so $a$ is of
type (Z2) (because it is not trivial) and therefore there is a
definable interval $I\ni a$ such that the structure induced on $I$
is linear, and every type extending $I$ is locally modular. Since
$p$ is a global type extending $I$ it must be locally modular.
\end{proof}

Observe that the last proof actually shows that any rich type is
not locally modular and that, in fact, a non-rich type over an
$\aleph_1$-saturated model is locally modular. We believe that a
direct proof (not using the trichotomy theorem for o-minimal
structures) of Remark \ref{rich} should exist, possibly giving the
result unconditionally (namely for all 1-types, not necessarily
over slightly saturated models). Despite this deficiency, this
result gives us a somewhat cleaner statement of the Trichotomy
theorem for o-minimal structures:

\begin{theorem}[Peterzil-Starchenko]\label{trichotomy}
Let $T$ be an o-minimal structure and $p$ a global 1-type, then
exactly one of the following hold:
\begin{enumerate}
\item $p$ is trivial.

\item $p$ is locally modular and not trivial, in which case it is
the generic type of a reduct, $\mathcal{V}$, of an ordered vector
space over an ordered division ring and $\mathcal V$ is definable.

\item $p$ is a generic type of a definable real closed field.
\end{enumerate}

\end{theorem}

\begin{remark}\label{porism}
The analogue of Remark \ref{rich} for minimal sets in stable
theories is well known (see Subsection 1.2.3 of \cite{Pi}). Being
local in nature, the proof automatically extends to seriously minimal
stable types (over algebraically closed sets) in the present
context, because of the uniqueness of non-forking extensions.
Consequently if $p(x)$ is a non locally modular seriously minimal
stable type there is some formula (over some localisation of $p$
at a generic $e$) $\phi(x_1,x_2;y_1,y_2)$ such that for any
$b_1,b_2\models p^{\otimes 2}$ and $c_1,c_2$ of $p\times p$ if
$\bar c\not\in \acl(\bar b)$ then there are infinitely many
elements in $\phi(x_1,x_2;b_1,b_2)\vartriangle
\phi(x_1,x_2;c_1,c_2)$ realizing $p\times p$. In particular,
\[
\{\exists y (\neg \phi(x,y;b_1,b_2)\vee \neg
\phi(x,y;c_1,c_2))\}\cup p
\]
is a non algebraic (and therefore the) non forking extension of
$p(x)$. By the uniqueness of the non-forking extension,
\[
p(x)\cup \left\{\exists y \left(\phi\left(x,y;b_1,b_2\right)\wedge
\phi\left(x,y;c_1,c_2\right)\right)\right\}
\]
is algebraic.
\end{remark}

The rest of this paper is devoted to proving a partial analogue of
the trichotomy of Theorem \ref{trichotomy} for theories
interpretable in o-minimal structures, though most of the work
will be done in the more general context of minimal types in
super-dependent theories.

\subsection{The geometry of minimal stable types}

The geometric analysis of minimal (and, more generally, regular
types) is a powerful tool in the investigation of stable theories.
In this section we show that the basic results of this analysis
remain valid for minimal stable types in dependent rosy theories.
Indeed (a posteriori) this is not surprising, as we will show that
in the geometrically non-trivial case minimal stable types are
seriously stable. Most of the proofs in this section vary between
automatic and straightforward adaptations of the corresponding
ones in the stable context.

Our first step is a generalisation of a theorem of Buechler (see
\S 1.3.1 of \cite{Pi}):

\begin{proposition}\label{seriously minimal}
If $p\in S(A)$ is a minimal non-trivial stable type in a super-dependent theory. Then $p(x)$ is
seriously minimal.
\end{proposition}
\begin{proof}
Let $p$ be as in the statement of the proposition. By a standard
argument we can reduce a minimal example of non-triviality of
$p(x)$ to a definable extension $p_1\supseteq p$ and realizations
$a,b,c\models p_1$ pairwise independent. We can also find a
formula $\phi(x,y,z)$ realized by $a,b,c$ and such that $a'\in
\acl(b'c')$, $b'\in \acl(c'a')$ and $c'\in \acl(a'b')$ for any
$a',b',c'\models \phi(x,y,z)$. By the Open Mapping Theorem, as
$p_1$ is a (stable) non-forking extension of $p$ if $\psi\in p_1$
is minimal there is $\sigma\in p$ which is a finite positive
Boolean combination of $A$-translates of $\psi$, and therefore
itself minimal. Thus it is enough to prove the claim for $p_1$; to
simplify the notation we will assume that $p=p_1$.

Notice that if some extension of $p$ to $\acl(A)$ is
seriously stable then so must be $p$. We can therefore assume
without loss of generality that $A=\acl(A)$.

By Fact 4.4 in \cite{EaOn} there exists $\alpha < \infty$ such
that $\tho(p)=\alpha$ and there is $\theta(x)\in p(x)$ with
$\tho(\theta(x))=\alpha$. Let
\[
 \chi(x,y):=\exists z (\phi(x,y,z)\wedge \theta(z))
\]
and let $\psi(y)$ be the $\chi$-definition of $\tp(a/Ab)$; by
Claim \ref{definability} $\psi$ is over $A$, and clearly $\models
\psi(b)$. We will prove that $\tho(\psi(y))=1$; i.e.
$\uth(b'/A)\leq 1$ for all $b'\models \psi$. So let $b'\models
\psi(y)$ and $a'\models \tp(a/Ab)|_{b'}$.

By definition $\models \chi(a',b')$ which implies the existence of
$c'\models \phi(a',b',z)\wedge \theta(z)$. Since $\tp(a'/Ab')|_A=
\tp(a'/A)$ Claim \ref{preserving th-rank} gives
$\alpha=\tho(a'/A)= \tho(a'/Ab')$. On the other hand $a'\in
\acl(c'b'A)$ and $c'\in \acl(a'b'A)$ so Proposition 4.6 in
\cite{EaOn} gives:
\[
\alpha=\tho(a'/A)= \tho(a'/Ab')= \tho(c'/Ab')\leq
\tho(c'/A)\le\alpha.
\]
Therefore $\tho(c'/Ab')= \tho(c'/A)$ and by \ref{preserving th-rank 2} $c'\thind_A b'$; it
follows that
\[
\uth(b'/A)=\uth(b'/Ac')=\uth(a'/Ac')\leq \uth(a'/A)=1.
\]
\end{proof}

\begin{remark}
 Easy examples (see \cite{HaOn}) show that the results of the last proposition are best possible. Trivial minimal stable types exist in theories where all definable sets are unstable.
\end{remark}

With this in hand, our aim is to show that non-trivial minimal
stable types are seriously stable. The work splits between the
locally modular case - to which the key is Hrushovski's
classification of locally modular regular types - and the non
locally modular case, which is solved by Buechler's dichotomy for
minimal types.

We start with Hrushovski's group recognition theorem in the
locally modular case. Although the proof translates word for word
into the present context, we give a general overview.
This is done not only for the sake of completeness and the
clarification of a few delicate points in the translation, but
mostly because we find it worthwhile pointing out tools appearing in
the proof, available in the present context, and potentially
useful in the future.

Recall that if $p$ is a stable type and $f$ a definable function
whose domain contains $p$ then the \emph{germ of $f$ at $p$} is
the class of $f$ under the (definable on families of definable
functions) equivalence relation $h\sim g$ if $h(a)=g(a)$ for all
$a\models p|h,g$. Note that for a stable type $p$ in a dependent
rosy theory, if $g$ is a germ, $a\models p|g$ then $g(a)$ is well
defined (i.e. does not depend on the choice of the
representative).

The first step is finding a type $p_0$ (of which $p$ is an almost
extension), and an invertible germ with domain $p_0$:

\begin{lemma}\label{type change}
Let $p$ be a minimal stable type in a dependent rosy theory. Assume that
$p$ is locally modular and non-trivial. Then there is a type $p_0$
of which $p$ is an almost extension, a type $q_0$ and an
invertible germ $\sigma:p_0\to q_0$ such that the type of $\sigma$
is minimal and stable.
\end{lemma}
\begin{proof}
We give some details of the proof. The key in the translation of
the present lemma into our context is noting that $p$ is an almost
extension of $p_0$ (and therefore the latter, as well as $\tp(\sigma)$,
are stable).

Adding constants to the language and extending $p$ accordingly we
may assume that in fact $p$ is modular. Since it is non-trivial we
can find $(a_1,a_2,e)$, pairwise independent realizations of $p$,
such that  $(a_1,a_2,e)$ is not independent. Obviously,
$q:=\stp(a_1,a_2/e)$ is stable so we can find $(b_1,b_2)\models
q|a_1a_2$. By modularity we can find $d,d_1\models p$ such that
$d\in \acl(a_1,b_2)\cap \acl(a_2,b_1)$ and
$d_1\in\acl(a_1,b_1)\cap\acl(d,e)$. Replacing $d_1$ with the code
of the (finite) set of its conjugates over $(a_1,a_2,b_1,b_2)$ we
may assume that $d\in \dcl(a_1,a_2,b_1,b_2)$. By a similar
argument, we may assume without loss that $(d,d_1)\in
\dcl(a_1,a_2,b_1,b_2,e)$.

Let $\bar d=(d,d_1)$, $\a=(a_1,a_2)$ and $\b=(b_1,b_2)$. Set
$p_0:=\tp(\b/e)$ and note that $p$ (or rather, its non-forking
extension) is an almost extension of $p_0$ (divide by the
equivalence relation $(b_1,b_2)\sim(b_1',b_2')$ if $b_1=b_1'$). So
we can find an $e$-definable function $f$ such that $\bar
d=f(\a,\b)$. Define an equivalence relation $x_1\sim_R x_2$ if
$f(y,x_1)=f(y,x_2)$ for $y\models \bar p|ex_1x_2$. By stability of
$p_0$ this is a definable equivalence relation, and because
$f(\a,\b')=\bar d$ implies that $\b'\in\acl(\b)$ it has finite
classes (on realizations of $p_0$). This gives us a family of
invertible germs $f(y,x)$ from $p_0/\sim_R$ to $r_0:=\stp(\bar
d/e)$. Setting $\sigma$ to be the germ of $f(x,\b')$ for a generic
$\b'\models p_0$ the lemma is proved, and since $\sigma$ is interalgebraic with $\b'$ (over $e$) we also get the stability of $\tp(\sigma/e)$.
\end{proof}

Now that the ground is set, and we do no longer have to change the
types with which we are working, the rest of the proof translates
automatically. We remind the remaining stages in the proof. Denote
$s_0:=\stp(\sigma/e)$ and $s:=\stp(\tau^{-1}\sigma/e)$ for
$\tau\models s|\sigma$. Observe that $s$ is a stable type: $s_0$
is stable - $p$ is an almost extension of $s_0$ - and therefore so
is $s_0^{-1}$, and also any type extending $s_0(x)\times
s_0^{-1}(y)$. We get a collection of germs of permutations of
$p_0$ forming a generic semi-group on $p_0$:

\begin{lemma}\label{final type}
$s$ is a stable type and $\ur(s/e)=1$. Moreover, if
$\sigma_1,\sigma_2$ are independent realizations of $s$ then
$\sigma_1\sigma_2\models s|\sigma_i$.
\end{lemma}

Once this is noted, the Hrushovski-Weil construction of the group
of germs out of a generic semigroup (on a stationary stable type)
goes through almost unaltered, and we get:

\begin{proposition}\label{ggroup}
Let $p$ be a stable type in a dependent rosy theory. Assume that
$p$ is stationary and that there exists a definable partial binary
function $*$ such that $a*b$ is defined for independent
$a,b\models p$ and satisfying:
\begin{enumerate}
 \item $*$ is generically transitive.
\item $a*b\models p|a$  and $a*b\models p|b$ for generic
$a,b\models p$.
\end{enumerate}
Then there exists a definable group $(G,\cdot)$ and a definable
embedding of $p$ into $G$ such that for independent $a,b$ the
image of $a*b$ is the $G$-product of the images of $a,b$.
Moreover, if $p$ is minimal, $G$ can be taken minimal.
\end{proposition}

The proof of Proposition \ref{ggroup} (in the stbale setting) is usually achieved in two
parts. In the first an $\infty$-definable (semi) group with the
desired properties is constructed, and then it is shown that an
$\infty$-definable group has a definable supergroup. In the present case, finding the
$\infty$-definable semi-group $G_0$ is a word by word translation
of the original proof. To show that it is in fact a group and to
get the definable supergroup $G\ge G_0$ we need to show that $G_0$
is stable, allowing us to use Hrushovski's original construction
of $G$ (Proposition 1.2 of \cite{Hr4}). That both $G_0$ and $G$
obtained above are stable, will follow from Lemma \ref{generic
stable} building on work in \cite{EaKrPi}. We start with some
basic definitions and facts from that paper.

\begin{definition}\label{definition th-generic}
Let $(G,\cdot)$ be a rosy group. A type $p(x)\in S(A)$ extending
$G(x)$ is \emph{\th-generic} if $a\cdot b\thind A,b$ and $b\cdot
a\thind A,b$ whenever $a,b\models G$ with $a\models p$ and
$a\thind_A b$.
\end{definition}

The following are the basic facts about \th-generic types which
can be found in Lemma 1.12 and Theorem 1.16 in \cite{EaKrPi}.

\begin{fact} \label{facts th-generic}
Let $(G,\cdot)$ be a rosy group defined over $\0$. Then the
following hold.

\begin{enumerate}
\item For any $A$ there is a \th-generic type of $G$ over $A$.

\item Let $a,b\in G$. If $\tp(a/A)$ is \th-generic and $b\in
\acl(A)$ then $\tp(a\cdot b/A)$ is \th-generic.

\item If $p\in S(A)$, $B\subseteq A$, and $\tp(a/A)$ is
\th-generic then so is $\tp(a/B)$.

\item If $a\thind_A b$ and $\tp(a/A)$ is \th-generic, then so is
$\tp(a/A,b)$.

\end{enumerate}
\end{fact}

\begin{lemma}\label{generic stable}
Let $(G,\cdot)$ be a rosy dependent type-definable group. Then the
following are equivalent.

\begin{enumerate}
\item Some \th-generic type $p(x)$ of $G$ is stable.

\item Every type $p(x)$ extending $G$ is stable.

\item $G$ is stable.
\end{enumerate}

In particular, If $G$ is a super dependent group then $G$ is super
stable if and only if some \th-generic type of $G$ is stable.
\end{lemma}

\begin{proof}
By definition and compactness a type-definable set in a dependent
theory is stable if and only if every complete type extending it
is stable, implying that (2) is equivalent to (3). The fact that
(2) implies (1) is immediate.

Suppose now that (1) holds. We may assume that $G$ is over $\0$;
let $p(x)$ be a \th-generic stable type extending $G$, and let
$q(x)$ be any type extending $G$. Replacing $p(x)$ with a non-\th-forking
extension we may assume, by Theorem \ref{unstable to stable}, that
$p$ and $q$ are over the same set of parameters, $A$.

Let $a\models p$ and let $b\models q$ be such that $a\thind_A b$.
By Fact \ref{facts th-generic} $\tp(a/A,b)$ is \th-generic and,
since $a$ and $ab$ are interdefinable over $A,b$, $\tp(a\cdot
b/A,b)$ is a stable \th-generic type of $G$ over $A,b$. But
$a\cdot b\thind_A b$ by hypothesis, so using Fact \ref{facts
th-generic} and Theorem \ref{unstable to stable}, we get that
$\tp(a\cdot b/A)$ is a \th-generic stable type of $G$ which
implies that $\tp(a\cdot b/A,a)$ is stable.

Since $a\cdot b$ and $a$ are interdefinable over $Aa$, $\tp(b/Aa)$
is stable; by symmetry $b\thind_A a$ so Theorem \ref{unstable to
stable} implies that $\tp(b/A)$ is a stable type.

\end{proof}

Thus, we have shown that Hrushovki's proof of the group
recognition theorem for locally modular minimal types translates
to minimal stable types in dependent theories. This allows us to
prove that, in fact, the analysis of non-trivial locally modular
types in dependent theories reduces to the same analysis in stable
theories:

\begin{proposition}\label{loc mod}
A stable minimal non trivial locally modular type is seriously
stable.
\end{proposition}

\begin{proof}
Let $p(x)$ be a stable non trivial locally modular type. Let $p_0$
be as provided by Lemma \ref{type change}, so $p_0$ is minimal and
stable and so is $s$, the type appearing in Lemma \ref{final
type}. Applying Proposition \ref{ggroup} to $s$ we find that there
is (using Lemma \ref{final type} again) a definable minimal group
$G$ and a definable embedding of $s$ into $G$, so $G$ has a
non-algebraic stable type. By Lemma \ref{generic stable} $G$ is
stable, so $s$ is seriously stable (because its embedding in $G$
is). Since orthogonality makes sense (in any theory) between
invariant (over small sets) types, it is clear that $p\not
\perp s$. Because both are stable, forking is equivalent to
\th-forking, so minimality implies that there is a definable
finite-to-finite correspondence between $p$ and $s$. Thus, $s$ is
seriously stable if and only if $p$ is, with the desired
conclusion.
\end{proof}

\begin{remark}
Most of the above translates quite easily to the more general
context of stable stably embedded locally modular regular types in
any theory. In particular, Hrushovski's construction, given a
locally modular regular $p$, of a type-definable group with
generic type domination equivalent to $p$ easily goes through (as
was pointed out to us by Hrushovski, one only has to notice that
all the parameters appearing in the proof can be taken from
$\dcl(P)$ - where $P$ is the set of realisations of $p$). As we do
not have in mind an immediate application of this observation we
do not give the details.
\end{remark}

We are now ready to prove the main result of this section:

\begin{theorem}\label{mainstable}
Let $p\in S(A)$ be a non-trivial stable minimal type. Then $p$ is
seriously stable and seriously minimal.
\end{theorem}

\begin{proof}
Since we already proved the theorem for the case where $p$ is
locally modular, we will assume this is not the case.

The proof is rather close to that of Proposition 2.3.2 in
\cite{Pi}. However, since some care is needed in the usage of both
forking and \th-forking we give the details.

Because $p$ is stable and any extension of $p$ to $\acl(A)$ is
non-forking, it follows (by the Open Mapping Theorem) that the
theorem holds of $p$ if and only if it holds of some (equivalently
any) extension of $p$ to $\acl(A)$. Thus, we may assume without
loss of generality that $A=\acl(A)$ (or, equivalently, that $p$ is
stationary). By Proposition \ref{seriously minimal} we know that
$p$ is seriously minimal so there is some formula $\mu(x)\in p(x)$
of \th-rank 1. It follows from the above discussion and Remark
\ref{porism} that we have a formula $\phi(\bar x, \bar
y)=\phi(x_1,x_2,y_1,y_2)$ over $A$ such that:

\begin{enumerate}
\item Without loss of generality $\phi(\bar x, \bar y)\to
(\mu(x_1)\land \mu(x_2)\land \mu(y_1)\land \mu(y_2))$. In
particular, if $\models \phi(\bar a, \bar b)$ then $\uth(\bar
a/A)\leq \tho(\bar a/A)\leq 2$ and $\uth(\bar b/A)\le \tho(\bar
b/A)\leq 2$.

\item For $\bar c, \bar b\models p\times p$ if $\bar b$ is generic
and $p\cup \{\exists y (\phi(x,y, \bar b)\wedge \phi(x,y, \bar
c))\}$ is a non forking extension of $p(x)$ then $\bar c\in \acl(A
\bar b)$.

\item For any $\bar a, \bar b$ if $\models \phi(\bar a, \bar b)$
then $a_i\in \acl(Aa_{3-i}\bar b)$ and $b_i\in \acl(A\bar a
b_{3-i})$ for $i\in \{1,2\}$. So by (1) we have
\[\uth(\bar b/A\bar a)\leq \tho(\bar b/A\bar a)\leq 1\] and
\[\uth(\bar a/A\bar b)\leq \tho(\bar a/A\bar b)\leq 1.\]
\end{enumerate}

Let $B\supseteq A$ be any algebraically closed set. Fix some
$b_1\models p(x)|B$ and $b_2\models p|Bb_1$ - the definable
extension of $p$ to $Bb_1$ - and denote $\bar b = (b_1,b_2)$. We can
now proceed as in Proposition 2.3.2 of \cite{Pi} and prove that
$\phi(x_1,x_2,\bar b)$ has finitely many non algebraic (global)
extensions and therefore has Morley rank 1.

\medskip

\noindent {\bf Claim I.} \emph{ If $\a \models \phi(\x, \bar b)$
and $\bar a\not\in \acl(B\bar b)$ then $\uth(\bar a/B)=2$ and
$\uth(\bar b/B\bar a)=1$.}

\begin{proof}
By assumptions we know that $\uth(\bar a/B\bar b)\geq 1$ and by
previous observations $\uth(\a/B\b)\leq 1$ so $\uth(\bar
a/B\bar b)=1$ and, being stable, this implies that $\ur(\bar
a/B\bar b)=1$. By assumption $\ur(\bar b/B)=2$ while $\ur(\b/B\a)\le 1$ so $\bar b \nind_B
\a$. Since $\tp(\bar b/B)$ is stable, this implies $\bar b\nthind_B
\bar a$ so
\[
2\geq \uth(\bar a/B)\geq \uth(\bar a/B\bar b)+1 \geq 2.
\]
This proves the first part of the claim. For the second part,
using Lascar's inequalities:
\[
\uth(\bar b/B)+\uth(\bar a/B\bar b)= \uth(\bar a\bar
b/B)=\uth(\bar a/B)+\uth(\bar b/B\bar a)
\]
But $\uth(\bar b/B)+\uth(\bar a/B\bar b)=3$ and $\uth(\bar a/B)=2$
so $\uth(\bar b/B\bar a)=1$.
\end{proof}

\medskip
\noindent {\bf Claim II.} \emph{ Suppose that $\bar a \models
\phi(x,\bar b)$, $\bar a \notin \acl(B\bar b)$ and let $\bar
c\models \tp(\bar b/B\bar a)|{B\bar b\bar a}$  (so that $\models
\phi(\bar a,\bar b)\wedge \phi(\bar a, \bar c)$). Then $\bar
c\thind_B\bar b$.}

\begin{proof}
Notice that  $\bar c\notin\acl (B\bar b)$ so we know $\a \in
\acl(B\b\bar c)$. On the other hand, by Claim I:
\[
\uth(\bar a \bar b \bar c/B)= \uth(\bar a/B)+\uth(\bar b/B\bar
a)+\uth(\bar c/B\bar a\bar b)=2+1+1
\]
so
\[
4= \uth(\bar c/B)+\uth(\bar b/B\bar c)+\uth(\bar a/B\bar b\bar
c)=2+\uth(\bar b/B\bar c)+0.
\]
whence $\uth(\bar b/B\bar c)=2$ as required.
\end{proof}

\noindent Let $\bar c\models\tp(\bar b/B)|{B \bar b}$  and let
\[P_{\bar c}:=\{ \tp(\bar a/\acl(B\bar b)) :\, \models \phi(\bar a,
\bar b)\wedge \phi(\bar a, \bar c)\}. \] Notice that $P_{\bar c}$
is finite by assumption, and by stationarity of $\tp(\bar b/B)$
independent of the choice of $\bar c$, so will be denoted $P$.

Now let $\bar a\models \phi(\bar x, \bar b)$ be such that $\bar a
\notin \acl(B\bar b)$ and $\bar d\models \tp(\bar b/B\bar a)|B\bar
b\bar a$. We get that $\tp(\bar d/B)=\tp(\bar c/B)$ and the last
claim assures that $\bar d\thind_B \bar b$. By stationarity
$\tp(\bar d/B\bar b)=\tp(\bar c/B\bar b)$ so $\bar a\models
\phi(\bar x,\bar b)\land \phi(\bar x,\bar d)$. Because $P$ does
not depend on the choice of $\bar c$ we get that $\tp(\a/B\b)\in
P$. Since $\a$ was arbitrary this shows that for any global
\th-minimal type $q(\bar x)\supseteq \{\phi(\bar x, \bar b)\}$,
the restriction $q\upharpoonright {\acl(B\bar b)}$ is in $P$.

So $\phi(\x,\b)$ has finitely many non-algebraic extensions to
$\acl(B\bar b)$. It will be enough, of course, to show that
$\phi(\x,\b)$ defines a stable set and has Morley rank 1 (because
then, obviously, so will be its projection on any of the
coordinates).

If we write $\phi(\bar x,\bar b)$ as $\phi(x,y,\bar b)$ and
$\psi(x,\bar b):=\exists y \phi(x,y,\bar b)$ then $\psi(x,\bar
b)\in p|{B\bar b}$ and $\psi(x,b)$ has finitely many extensions to
$\acl(B\bar b)$. Since $\psi(x,\b)\in p|B\b)$ it follows that
there is some $\theta(x)\in p$ which is a finite positive Boolean
combination of $A$-translates of $\psi(x,\b)$, namely
$\theta(x)\equiv \bigvee_l \bigwedge_k \psi(x,\bar b_{l,k})$. By
what we have shown, each of $\psi(x,\bar b_{l,k})$ has finitely
many completions in $\acl(B\b_{l,k})$ so $\theta(x)$ has finitely
many completions in $B$. Since $B$ was arbitrary, \ref{stable types} implies that
$\theta(x)$ is stable. Since it only has finitely many generic
types, it must be of finite Morley rank.
\end{proof}

\section{minimal types in theories interpretable in o-minimal
structures}\label{in o-minimal}

We can now collect the results of Section \ref{geometry-section}, together
with the main results of \cite{HaOn} to give our (incomplete)
version of Theorem \ref{trichotomy} for structures
interpretable in o-minimal theories:

\begin{theorem}\label{main}
 Let $\CM$ be a structure definable in an o-minimal theory.
 Assume $\CM$ is $\aleph_1$-saturated and let $p(x)\in S_1(M)$
 be a minimal non-trivial type then:
\begin{itemize}
\item If $p(x)$ is unstable it is it is almost o-minimal.
Moreover, if $p(x)$ is rich it has an almost extension extending a
real closed field. Otherwise, any o-minimal almost extension of
$p(x)$ is definable in an ordered vector space over an ordered
division ring. \item If $p(x)$ is stable it is seriously stable.
In particular, if it is not locally modular it is strongly
minimal. Otherwise there exists a stable minimal type-definable
group $G$ acting regularly on some minimal $s\not \perp p$.
\end{itemize}
\end{theorem}

Recall that in \cite{HaOnPe} we conjectured that Zilber's Principle (asserting,
roughly, that rich structures must arise from definable fields) holds in any
geometric structure interpretable in an o-minimal theory. Restricting ourselves
to structures \emph{definable} in o-minimal theories, our theorem reduces the
problem to Zilber's Trichotomy for strongly minimal structures definable in
o-minimal theories:

\begin{conjecture}[Zilber's Trichotomy conjecture]
 Let $\CN$ be a strongly minimal theory interpretable in an o-minimal
 structure. $\CN$ is not locally modular if and only if it interprets
 an algebraically closed field.
\end{conjecture}

\medskip

We start the proof of the theorem with an easy corollary to the
work done until now, answering a question of our own (see
\cite{HaOn}).

\begin{proposition}\label{o-min}
Let $p$ be an unstable minimal type definable in a theory
interpretable in an o-minimal structure. Then $p$ is almost
o-minimal, namely there is o-minimal almost extension of $p$.
\end{proposition}
\begin{proof}
By Theorem 2 of \cite{HaOn} $p$ is either almost o-minimal or has
a stable non-algebraic extension. By Theorem \ref{unstable to
stable} any stable extension of $p$ is a \th-forking one. Since
$p$ is minimal, and therefore has no non-algebraic \th-forking
extensions, the proposition is proved.
\end{proof}

The last proposition provides us, for a minimal unstable type, with an
o-minimal almost extension of $q$, but as the richness (or non-triviality)
of $q$ may need outer parameters to be witnesses, does not a priori give
information about the geometry of the o-minimal structure of which $q$
is generic. For that purpose we need the following:

\begin{lemma}\label{embedded}
 Let $\CN$ be a super-rosy structure of finite \th-rank. Let $Z$ be an
 $\CN$-definable set such that the structure $\mathcal Z$ induced on
 $Z$ is o-minimal. Let $p$ be any complete type extending $Z$. Then
 $p$ is rich (non-trivial) as a type in $\CN$ if and only if it is
 rich (non-trivial) as a type in $\mathcal Z$.
\end{lemma}
\begin{proof}
We only need to prove the left to right direction. Let $C(x,y;\bar
z)$ be any normal $\CN$-definable family of plane curves in $Z^2$.
It will suffice to show that if $\tho((\exists^{\infty}
x,y)C(x,y;\bar z))=n$ then there exists a $\mathcal Z$-definable
subfamily of $C(x,y,;\bar z)$ of o-minimal dimension $n$. We may
assume (otherwise the claim is trivial) that $n>0$ and that if
$(\exists^{\infty} x,y)C(x,y;\a)$ and $\uth(a/\0)=n$ then $C(x,y)$
is not an almost constant curve (namely, for generic $(c,d)\models
C(x,y;\a)$ both $d\thind \a$ and $c\thind \a)$.

We may assume that $C(x,y;\bar z)$ is $\0$-definable (in $\mathcal
Z$); let $\bar a$ be such that $(\exists^{\infty}x,y)C(x,y;\a)$
and $\tho(a/\0)=n$. By coordinatisation there exists some
$B\subseteq N$ and $a_1,\dots,a_n$ such that
\begin{itemize}
 \item $a\thind B$,
\item $\uth(a_1/B)=1$,
\item $\uth(a_{i+1}/Ba_1,\dots,a_i)=1$ for all $i\ge 1$ and
\item $a_n=\a$.
\end{itemize}
We may also assume that all the $\CN$-definable subsets of $Z$
defined over $B$ are $\mathcal Z$-atomic, so without loss of
generality $B=\0$ as well. Let $(c,d)\models C(x,y,\a)$ for some
generic $c\in Z$, so $d\in Z$ and $d\in \acl(c,a_1,\dots,a_n)$.
Adding constants to the language, we may minimise $n$ to assure
that $d\notin \acl(c,a_1,\dots,a_{n-1})$ so - as
$\uth(a_n/a_1,\dots,a_{n-1})=1$ and $a_n\thind_{a_1,\dots,a_{n-1}}
c$ - we conclude that $a_n\in
\acl(c,d,a_1\dots,a_{n-1})$. By induction on $n$, it follows that
we can find $\bar c, \bar d\subseteq Z$ such that $C(x,y;\a)$ is
definable over $\bar c, \bar d$. Since the same is true of any
$a'\equiv a$, we get the desired conclusion.
\end{proof}

\begin{remark}
The proof of the previous lemma is based on the proof of Lemma 2.3
of \cite{PeSt}. It is not clear to us, however, whether Lemma
\ref{embedded} can be strengthened to give the full result of
Peterzil and Starchenko and assure that a closed interval in an
o-minimal set definable in a super-rosy theory of finite rank is
stably embedded.
\end{remark}

We can now proceed to proving the Theorem \ref{main}:

\medskip

Assume first that $p$ is locally modular. If it is stable
then by Theorem \ref{mainstable} it is seriously stable and
seriously minimal, and the result now follows from the analogous
statement in the stable context (see Theorem 5.1.1 of \cite{Pi})
and the fact that stable sets in dependent theories are stably
embedded. So we may assume that $p$ is unstable.

By Remark \ref{easy facts}(4) every non-algebraic extension of $p$
is non-trivial. In particular, $p$ has an almost o-minimal
extension $q$ (by Proposition \ref{o-min}). Being non-trivial, so
is every non-algebraic almost extension of $q$. So there is a
finite equivalence relation $\sim$ such that $q/\sim$ is a generic
type in some definable o-minimal set $\mathcal Z$. By Lemma
\ref{embedded} $q/\sim$ is non-trivial, and since $p$ were not
rich, neither is $q/\sim$. The result now follows from the
Trichotomy theorem for o-minimal structures.

\medskip

Now, let $p$ be a non locally modular type. If $p$ is stable then
by Theorem \ref{mainstable} it is seriously stable and by
Buechler's dichotomy it is strongly minimal. So we may assume that
$p$ is unstable. In that case, let $p_0:=p\upharpoonright
\acl(\0)$ and $a\models p$ any realization. By Proposition
\ref{o-min}, $p_0$ has an almost o-minimal extension $q_0$. Using
automorphisms we may assume that $a\models q_0$. Because
$p\supseteq q_0$ it must be that $q_0$ is not locally modular. Let
$\sim$ be a definable equivalence relation with finite classes
such that $q:=q_0/\sim$ is o-minimal. Since $q_0$ is not locally
modular neither is $q$, and it is not locally modular also as a an
o-minimal type. By Remark \ref{rich} $q$ is rich (as an o-minimal
type). So by the Trichotomy Theorem for o-minimal structures there
is a definable real closed field, of which $q$ is an extension.
Thus, $p/\sim$ is an almost extension of $p$ extending a definable
real closed field, as required.

\medskip

\noindent This finishes the proof of Theorem \ref{main}.\qed

\medskip

We can now sum up our findings by:

\begin{corollary}
Let $\CM$ be a structure definable in an o-minimal theory. Then:
\begin{enumerate}
\item  $\CM$ interprets a real closed field if and only if there
is a minimal non-locally modular (non trivial) global type.

\item Assuming Zilber's Trichotomy conjecture $\CM$ interprets a
pure algebraically closed field if and only if it has a minimal
non-locally modular stable type.
\end{enumerate}
\end{corollary}

 \bibliographystyle{abbrv}
 \bibliography{all}
\end{document}